\documentclass[12pt]{amsart}

\usepackage{mathrsfs,amssymb}

\newtheorem{theorem}{Theorem}[section]
\newtheorem{lemma}[theorem]{Lemma}
\newtheorem{prop}[theorem]{Proposition}

\theoremstyle{definition}
\newtheorem{definition}[theorem]{Definition}

\newtheorem{con}[theorem]{Conjecture}

\theoremstyle{remark}

\numberwithin{equation}{section}

\let \la=\lambda
\let \e=\varepsilon
\let \d=\delta
\let \o=\omega
\let \a=\alpha

\let \b=\beta

\let \O=\Omega
\let \si=\sigma

\begin{document}
\title[estimate of Calder\'on-Zygmund operators]
{On an estimate of Calder\'on-Zygmund operators by dyadic positive operators}

\author{Andrei K. Lerner}
\address{Department of Mathematics,
Bar-Ilan University, 52900 Ramat Gan, Israel}
\email{aklerner@netvision.net.il}

\begin{abstract}
Given a general dyadic grid ${\mathscr{D}}$ and a sparse family of cubes ${\mathcal S}=\{Q_j^k\}\in {\mathscr{D}}$, define a dyadic positive operator
${\mathcal A}_{{\mathscr{D}},{\mathcal S}}$ by
$${\mathcal A}_{{\mathscr{D}},{\mathcal S}}f(x)=\sum_{j,k}f_{Q_j^k}\chi_{Q_j^k}(x).$$
Given a Banach function space $X({\mathbb R}^n)$ and the maximal Calder\'on-Zygmund operator $T_{\natural}$, we show that
$$\|T_{\natural}f\|_X\le c(n,T)\sup_{{\mathscr{D}},{\mathcal S}}\|{\mathcal A}_{{\mathscr{D}},{\mathcal S}}|f|\|_{X}.$$

This result is applied to weighted inequalities. In particular, it implies: (i) the ``two-weight conjecture" by D. Cruz-Uribe and C.~P\'erez
in full generality; (ii) a simplification of the proof of the ``$A_2$ conjecture"; (iii) an extension of certain mixed $A_p$-$A_r$ estimates to general
Calder\'on-Zygmund operators; (iv) an extension of sharp $A_1$ estimates (known for $T$) to the maximal Calder\'on-Zygmund operator~$T_{\natural}$.

\end{abstract}



\maketitle

\section{Introduction}

A Calder\'on-Zygmund operator in ${\mathbb R}^n$ is an $L^2$ bounded integral operator with kernel $K$
satisfying the following growth and smoothness conditions:
\begin{enumerate}
\renewcommand{\labelenumi}{(\roman{enumi})}
\item
$|K(x,y)|\le \frac{c}{|x-y|^n}$ for all $x\not=y$;
\item
there exists $0<\d\le 1$ such that
$$|K(x,y)-K(x',y)|+|K(y,x)-K(y,x')|\le
c\frac{|x-x'|^{\d}}{|x-y|^{n+\d}},$$ whenever $|x-x'|<|x-y|/2$.
\end{enumerate}

Given a Calder\'on-Zygmund operator $T$, define its maximal truncated version by
$$T_{\natural}f(x)=\sup_{0<\e<\nu}\Big|\int_{\e<|y|<\nu}K(x,y)f(y)dy\Big|.$$

By a {\it general dyadic grid} ${\mathscr{D}}$ we mean a collection of
cubes with the following properties: (i)
for any $Q\in {\mathscr{D}}$ its sidelength $\ell_Q$ is of the form
$2^k, k\in {\mathbb Z}$; (ii) $Q\cap R\in\{Q,R,\emptyset\}$ for any $Q,R\in {\mathscr{D}}$;
(iii) the cubes of a fixed sidelength $2^k$ form a partition of ${\mathbb
R}^n$.

We say that $\{Q_j^k\}\in {\mathscr{D}}$ is a {\it sparse family} of dyadic cubes if:
(i)~the cubes $Q_j^k$ are disjoint in $j$, with $k$ fixed;
(ii) if $\Omega_k=\cup_jQ_j^k$, then $\Omega_{k+1}\subset~\Omega_k$;
(iii) $|\Omega_{k+1}\cap Q_j^k|\le \frac{1}{2}|Q_j^k|$.

Given a dyadic grid ${\mathscr{D}}$ and a sparse family ${\mathcal S}=\{Q_j^k\}\in {\mathscr{D}}$, consider
a dyadic positive operator ${\mathcal A}$ defined by
$${\mathcal A}f(x)={\mathcal A}_{{\mathscr{D}},{\mathcal S}}f(x)=\sum_{j,k}f_{Q_j^k}\chi_{Q_j^k}(x)$$
(we use the standard notation $f_Q=\frac{1}{|Q|}\int_Qf$).

Our main result is the following.

\begin{theorem}\label{mainr} Let $X$ be a Banach function space over ${\mathbb R}^n$ equipped with Lebesgue measure.
Then, for any appropriate $f$,
$$
\|T_{\natural}f\|_{X}\le c(T,n)\sup_{{\mathscr{D}},{\mathcal S}}\|{\mathcal A}_{{\mathscr{D}},{\mathcal S}}|f|\|_{X},
$$
where the supremum is taken over arbitrary dyadic grids ${\mathscr{D}}$ and sparse families ${\mathcal S}\in {\mathscr{D}}$.
\end{theorem}

We consider several applications of this result in the case when $X$ is the weighted Lebesgue space, $X=L^p(u)$ (by a weight we mean
a non-negative locally integrable function).

The operators similar to ${\mathcal A}$ were used in \cite{CMP,L,L2} to deal with several classical transforms represented in terms
of the Haar shift operators of bounded complexity (for example, the Hilbert, Riesz and Beurling transforms). Now, by Theorem \ref{mainr},
we have that the results obtained by this approach hold for arbitrary Calder\'on-Zygmund operators. In particular, we mention the work \cite{CMP} by D. Cruz-Uribe, J.~ Martell and
C.~P\'erez where it was found a very simple proof of both the ``two-weight" and ``$A_2$" conjectures for ${\mathcal A}$ (and hence for the above mentioned classical
operators). Now we have that this proof is automatically extended to any Calder\'on-Zygmund operator, and, in particular, this yields the ``two-weight conjecture"
due to D. Cruz-Uribe and C. P\'erez in full generality. Moreover, the approach to ${\mathcal A}$ from \cite{CMP} allows actually to get a rather general sufficient condition for the two-weighted boundedness of $T$.

First we observe that the two-weighted estimates for dyadic positive operators (and in particular for ${\mathcal A}$)
have been recently characterized by M.~Lacey, E. Sawyer and I. Uriarte-Tuero \cite{LSU}; a necessary and sufficient condition is expressed in
terms of Sawyer-type testing conditions. We mention here a different simple characterization which is partially based on an idea used in \cite{CMP} to deal with ${\mathcal A}$. Its advantage is that it avoids the use of the notions of ${\mathscr{D}}$ and ${\mathcal S}$. On the other hand, it requires the following bi(sub)linear maximal operator defined by
$${\mathcal M}(f,g)(x)=\sup_{Q\ni x}\left(\frac{1}{|Q|}\int_Q|f|\right)\left(\frac{1}{|Q|}\int_Q|g|\right),$$
where the supremum is taken over arbitrary cubes $Q$ containing the point $x$.

\begin{theorem}\label{difch} Let $1<p<\infty$ and let $u,v$ be arbitrary weights. Then the following equivalence
$$
\sup_{{\mathscr{D}},{\mathcal S}}\|{\mathcal A}_{{\mathscr{D}},{\mathcal S}}\|_{L^p(v)\to L^p(u)}\asymp
\|{\mathcal M}\|_{L^p(v)\times L^{p'}(u^{1-p'})\to L^1}
$$
holds with the corresponding constants depending only on $n$.
\end{theorem}

In order to give a general formulation of the two-weighted Muckenhoupt type sufficient condition for $T$, we invoke again the Banach function space $X$.
Given a cube $Q$, define the $X$-average of $f$ over $Q$ and the maximal operator $M_X$ by
$$\|f\|_{X,Q}=\|(f\chi_Q)(\ell_Q\cdot)\|_X,\quad M_Xf(x)=\sup_{Q\ni x}\|f\|_{X,Q}.$$
This operator was introduced and studied by C. P\'erez \cite{P0,P}. By $X'$ we denote the associate space to $X$.

Theorems \ref{mainr} and \ref{difch} easily imply the following.

\begin{theorem}\label{twosuf} Let $1<p<\infty$, and let $X$ and $Y$ be the Banach function spaces such that $M_{X'}$ and $M_{Y'}$ are bounded on $L^{p'}$ and $L^{p}$,
respectively. Then
\begin{equation}\label{twoestim}
\|T_{\natural}\|_{L^p(v)\to L^p(u)}\le c\sup_Q\|u^{1/p}\|_{X,Q}\|v^{-1/p}\|_{Y,Q}.
\end{equation}
\end{theorem}

Assume that $X=L^p$. Then $M_{X'}=M_{L^{p'}}$. The operator $M_{L^{p'}}$ is not bounded on $L^{p'}$
since this would be equivalent to the boundedness of $M$ on $L^1$. Similarly, if $Y=L^{p'}$, then $M_{Y'}=M_{L^p}$ is not bounded on $L^{p}$.
It is natural that in the case $X=L^p,Y=L^{p'}$ the condition of the theorem is not satisfied since in this case
the finiteness of the right-hand side of (\ref{twoestim}) means that a couple $(u,v)$ satisfies the $A_p$ Muckenhoupt
condition. But it is well known that $(u,v)\in A_p$ is not sufficient even for the two-weighted boundedness of the Hardy-Littlewood maximal operator~$M$~\cite{S}.
On the other hand, taking the $X$ and $Y$ averages on the right-hand side of (\ref{twoestim}) a bit bigger than the $L^p$ and $L^{p'}$ averages, the corresponding
operators $M_{X'}$ and $M_{Y'}$ will be a bit smaller than $M_{L^{p'}}$ and $M_{L^p}$, and we obtain their boundedness on $L^{p'}$ and $L^{p}$,
respectively, and therefore a sufficient two-weighted condition.

A typical situation occurs when $X=L^{A}$ is the Orlicz space defined by means of the Young function $A$, equipped with the Luxemburg norm. In this case $X'=L^{\bar A}$
(with the equivalence of norms), where $\bar A$ is the Young function complementary to $A$. The boundedness of $M_{L^{A}}$ on $L^p$ was characterized by C. P\'erez \cite{P}; a necessary and sufficient condition is the $B_p$ condition which says that for some $c>0$,
$$\int_c^{\infty}\frac{A(t)}{t^p}\frac{dt}{t}<\infty.$$
Hence, if $X=L^{A}$ and $Y=L^{B}$, the boundedness of $M_{X'}$ and $M_{Y'}$ on $L^{p'}$ and $L^{p}$ is equivalent to that
$\bar A\in B_{p'}$ and $\bar B\in B_{p}$. In this case Theorem \ref{twosuf} yields the ``two-weight conjecture" by D.~Cruz-Uribe and C. P\'erez
mentioned above (we use the notation $\|f\|_{L^A,Q}=\|f\|_{A,Q}$).

\begin{con}\label{two-weight} {\it Given $p,1<p<\infty,$ let $A$ and $B$ be two Young functions such that
$\bar A\in B_{p'}$ and $\bar B\in B_{p}$. If a couple of weights $(u,v)$ satisfies
$$
\sup_Q\|u^{1/p}\|_{A,Q}\|v^{-1/p}\|_{B,Q}<\infty,
$$
then
$$\|Tf\|_{L^p(u)}\le c\|f\|_{L^p(v)}.$$
}
\end{con}

For a complete history of this conjecture and partial results we refer to a recent book \cite{CMP2}.
Under certain restrictions on $A$ and $B$, the conjecture was proved in \cite{CMP1}.
By means of the ``local mean oscillation decomposition" the conjecture was proved for any $T$ in the case $p>n$ in~\cite{L1}.
After that, using the same decomposition and the operator ${\mathcal A}$,
the conjecture was proved for the Hilbert, Riesz and Beurling transforms in \cite{CMP}.
In a recent work \cite{NRTV}, the conjecture was completely proved in the
case $p=2$ by means of the Bellman function method. Also, in \cite{CRV}, the conjecture was proved for
the so-called $\log$ bumps and for certain $\log\log$ bumps.

The rest of applications of Theorem \ref{mainr} are given in Section 2 below.
We turn now to the main ingredients used in the proof of this theorem.

\begin{list}{\labelitemi}{\leftmargin=1em}
\item
A representation of $T$ in terms of the Haar shift operators ${\mathbb S}={\mathbb S}_{{\mathscr{D}}}^{m,k}$
obtained by T. Hyt\"onen \cite{H} (see also \cite{H1,HPTV}), and its ``maximal truncated"
corollary proved in \cite{HLMORSU}.
\item
A recent estimate by T. Hyt\"onen and M. Lacey \cite{HL} where they used the ``local mean oscillation decomposition" from \cite{L1} to bound
${\mathbb S}_{{\mathscr{D}}}^{m,k}$ by the sum $\sum_{i=1}^{\kappa+1}{\mathcal A}_i$, where $\kappa=\max(k,m,1)$ is the complexity of ${\mathbb S}_{{\mathscr{D}}}^{m,k}$, and the operators ${\mathcal A}_i$ are defined by
$${\mathcal A}_if(x)={\mathcal A}_{{\mathscr{D}},{\mathcal S},i}f(x)=\sum_{k,j}f_{(Q_j^k)^{{(i)}}}\chi_{Q_j^k}(x)$$
(here $Q^{(i)}$ denotes the $i$-th ancestor of $Q$, that is, the unique dyadic cube containing $Q$ and such that ${\ell}_{Q^{(i)}}=2^i{\ell}_Q$).
Observe that the idea to bound ${\mathbb S}_{{\mathscr{D}}}^{m,k}$ by operators ${\mathcal A}_i$ goes back to \cite{CMP}.
But a crucial point is the linear dependence on ${\kappa}$ in \cite{HL},
while it was exponential in \cite{CMP}.
\item
The key idea in \cite{HL} was that ${\mathcal A}_i$ can be viewed as a Haar shift operator of complexity $i$,
but with a positive kernel. This fact allowed to simplify certain arguments used when dealing with general Haar shift operators.
Our novel point in this paper is that one can use again the ``local mean oscillation decomposition" to bound ${\mathcal A}_i$.
More precisely, we consider the formal adjoint of ${\mathcal A}_i$ given by
$${\mathcal A}_i^{\star}f(x)=\frac{1}{2^{in}}\sum_{k,j}f_{Q_j^k}\chi_{(Q_j^k)^{(i)}}(x).$$
We show that given a finite sparse family ${\mathcal S}_1$, there is a sparse family ${\mathcal S}_2$ such that for a.e. $x$,
\begin{equation}\label{esa}
{\mathcal A}_{{\mathcal S}_1,i}^{\star}f(x)\le c(n)i\big(Mf(x)+{\mathcal A}_{{\mathcal S}_2}f(x)\big).
\end{equation}
Combining this estimate with the above mentioned ingredients leads easily to Theorem \ref{mainr}.
\end{list}

Observe that the ``local mean oscillation decomposition" proved in \cite{L1} states that
\begin{equation}\label{de}
|f(x)-m_f(Q_0)|\le
4M_{1/4;Q_0}^{\#,d}f(x)+4\sum_{k,j}
\o_{\frac{1}{2^{n+2}}}(f;(Q_j^k)^{(1)})\chi_{Q_j^k}(x)
\end{equation}
(see Section 4 below for the definitions of the objects involved here). This estimate would allow to get (\ref{esa}) with ${\mathcal A}_1$ instead of
${\mathcal A}={\mathcal A}_0$ on the right-hand side, and, as a result, we would get Theorem \ref{mainr} with~${\mathcal A}_1$. This is not actually important
from point of view of main applications. But in order to arrive to a smaller operator ${\mathcal A}$, we will use the following variant of (\ref{de}) proved
in Theorem~\ref{decom1} below:
$$
|f(x)-m_f(Q_0)|\le
4M_{\frac{1}{2^{n+2}};Q_0}^{\#,d}f(x)+2\sum_{k,j}
\o_{\frac{1}{2^{n+2}}}(f;Q_j^k)\chi_{Q_j^k}(x).
$$
The main difference with (\ref{de}) is that the oscillations here are taken over the cubes $Q_j^k$.

The paper is organized as follows. In Section 2 we give some other applications of Theorem \ref{mainr}. Section 3 contains basic facts
concerning the Haar shift operators. In Section 4 we prove the above mentioned version of the ``local mean oscillation decomposition".
Theorem \ref{mainr} is proved in Section 5. Finally, in Section 6 we prove Theorems \ref{difch} and~\ref{twosuf}.

Throughout the paper we will use the following notation. Given a sparse family $\{Q_j^k\}$, set $E_j^k=Q_j^k\setminus \Omega_{k+1}$.
Observe that the sets $E_j^k$ are pairwise disjoint and $|Q_j^k|\le 2|E_j^k|$. In the case when the argument does not depend on a particular
grid ${\mathscr{D}}$ and a sparse family ${\mathcal S}\in {\mathscr{D}}$ we drop the subscripts ${\mathscr{D}}$ and ${\mathcal S}$, and we will
assume that ${\mathscr{D}}$ is the standard dyadic grid.

\section{Applications}
\subsection{The ``$A_2$ conjecture"} Given a weight $w$, define its
$A_p$ characteristic by
$$\|w\|_{A_p}\equiv\sup_QA_p(w;Q)=\left(\frac{1}{|Q|}\int_Qw\,dx\right)
\left(\frac{1}{|Q|}\int_Qw^{-\frac{1}{p-1}}\,dx\right)^{p-1},$$
The ``$A_2$ conjecture" states that for a Calder\'on-Zygmund operator $T$,
\begin{equation}\label{a2}
\|T\|_{L^p(w)}\le
c(T,p,n)\|w\|_{A_p}^{\max\big(1,\frac{1}{p-1}\big)}\quad(1<p<\infty).
\end{equation}
Note that by extrapolation it suffices to get this result in the case $p=2$ (this explains the name of the conjecture).
In its full generality this conjecture was recently settled by T. Hyt\"onen~\cite{H} (see also \cite{H1,HPTV}). Soon after that it was shown in
\cite{HLMORSU} that (\ref{a2}) holds for $T_{\natural}$ as well. The proof of (\ref{a2}) is based on the
representation of $T$ in terms of the Haar shift operators ${\mathbb S}_{{\mathscr{D}}}^{m,k}$. After that the proof reduces to showing (\ref{a2}) for
${\mathbb S}_{{\mathscr{D}}}^{m,k}$ in place of $T$ with the corresponding constant depending linearly (or polynomially) on the complexity.
Observe that over the past year several different proofs of the latter step appeared (see, e.g., \cite{La,T}).

We now have that (\ref{a2}) follows immediately from Theorem \ref{mainr} combined with the estimate
\begin{equation}\label{estcmp}
\|{\mathcal A}\|_{L^2(w)}\le c(n)\|w\|_{A_2}
\end{equation}
proved in \cite{CMP}. The proof of (\ref{estcmp}) is quite elementary, and we give it here for the sake of
the completeness. Let $M^d_w$ be the dyadic weighted maximal operator; we use that it is bounded on $L^p(w)$ with the bound independent of $w$.
Assuming that $f,g\ge 0$, by H\"older's inequality we have
\begin{eqnarray*}
\int_{{\mathbb R}^n}({\mathcal A}f)gdx&=&\sum_{j,k}f_{Q_j^k}g_{Q_j^k}|Q_j^k|\\
&\le& 2\sum_{j,k}A_2(w;Q_j^k)\Big(\frac{1}{w^{-1}(Q_j^k)}\int_{Q_j^k}f\Big)
\Big(\frac{1}{w(Q_j^k)}\int_{Q_j^k}g\Big)|E_j^k|\\
&\le& 2\|w\|_{A_2}\sum_{j,k}\int_{E_j^k}M^d_{w^{-1}}(fw)M^d_w(gw^{-1})dx\\
&\le& 2\|w\|_{A_2}
\int_{{\mathbb R}^n}M^d_{w^{-1}}(fw)M^d_w(gw^{-1})dx\\
&\le& 2\|w\|_{A_2}\|M^d_{w^{-1}}(fw)\|_{L^2(w^{-1})}\|M^d_w(gw^{-1})\|_{L^2(w)}\\
&\le& c\|w\|_{A_2}\|f\|_{L^2(w)}\|g\|_{L^2(w^{-1})},
\end{eqnarray*}
which yields (\ref{estcmp}) by duality.

Resuming, the ``$A_2$ conjecture" follows now from the next steps.
\begin{list}{\labelitemi}{\leftmargin=1em}
\item A representation of $T$ in terms of ${\mathbb S}_{{\mathscr{D}}}^{m,k}$ \cite{H,H1,HPTV}.
\item The ``local mean oscillation decomposition" bound of ${\mathbb S}_{{\mathscr{D}}}^{m,k}$ by the operators ${\mathcal A}_i$ \cite{HL}.
\item The ``local mean oscillation decomposition" bound of ${\mathcal A}^{\star}_i$ by ${\mathcal A}$.
\item The $L^2(w)$ bound of ${\mathcal A}$ \cite{CMP}.
\end{list}

\subsection{Mixed $A_p$-$A_{\infty}$ estimates}
Given a weight $w$, define its $A_{\infty}$ characteristic by
$$\|w\|_{A_{\infty}}=\sup_{Q\subset {\mathbb R}^n}\frac{1}{w(Q)}\int_QM(w\chi_Q).$$
Note that $\|w\|_{A_{\infty}}\le c(p,n)\|w\|_{A_p}$ for any $p>1$.

M. Lacey \cite{L} showed that for classical singular integrals (\ref{a2}) can be improved as follows
\begin{equation}\label{Lac}
\|T_{\natural}\|_{L^p(w)}\le
c(T,p,n)\|w\|_{A_p}^{1/p}\max\Big((\|w\|_{A_{\infty}}')^{1/p'}
,(\|\si\|_{A_{\infty}}')^{1/p}\Big),
\end{equation}
where $\si=w^{1-p'}$. Also, it was conjectured in \cite{L} that this estimate holds for any Calder\'on-Zygmund operator.
Soon after that the conjecture was proved in \cite{HL}; the proof was based on the analysis of the operators ${\mathcal A}_i$.
On the other hand, the proof in \cite{L} was based on showing (\ref{Lac}) for ${\mathcal A}$ in place of $T_{\natural}$. Hence, by Theorem~\ref{mainr}
we have that this proof actually yields (\ref{Lac}) in the general case.

\subsection{Mixed $A_p$-$A_{r}$ estimates}
Given a weight $w$, define its mixed $A_p$-$A_r$ characteristic by
$$\|w\|_{(A_p)^{\a}(A_r)^{\b}}=\sup_{Q\subset {\mathbb R}^n}A_p(w;Q)^{\a}A_r(w;Q)^{\b},$$
where $\a,\b\ge0$.

In \cite{L2}, it was proved that for any $2\le p\le r<\infty$,
\begin{equation}\label{esti1}
\|{\mathcal A}\|_{L^p(w)}\le
c(p,r,n)\|w\|_{(A_p)^{\frac{1}{p-1}}(A_r)^{1-\frac{1}{p-1}}}.
\end{equation}
By duality, it follows from this that for any $1<p<2$ and $r>p'$,
\begin{equation}\label{esti2}
\|{\mathcal A}\|_{L^p(w)}\le c(p,r,n)\|\si\|_{(A_{p'})^{\frac{1}{p'-1}}(A_r)^{1-\frac{1}{p'-1}}}
\end{equation}

From this, estimates (\ref{esti1}) and (\ref{esti2}) were obtained in \cite{L2} for classical singular integrals in place of ${\mathcal A}$. Now, by Theorem \ref{mainr}
we have that they hold for any Calder\'on-Zygmund operator $T$ (and $T_{\natural}$). Note that the difference between these estimates and (\ref{Lac}) is that
in the mixed $A_p$-$A_r$ characteristic only one supremum is involved, while the right-hand side of (\ref{Lac}) involves two independent suprema. It was shown in \cite{L2} by simple examples that the right-hand sides in (\ref{esti1}) and (\ref{Lac}) are incomparable. In \cite{HL}, a new conjecture was posed about the $L^p(w)$ bound for $T$ implying the estimates of both types. By Theorem \ref{mainr} we have that it suffices to prove this conjecture for~${\mathcal A}$. However, even for this simple operator the new conjecture is still not clear.

\subsection{Sharp $A_1$ estimates}
Recall that $w$ is an $A_1$
weight if there exists $c>0$ such that $Mw(x)\le cw(x)$ a.e.; the
smallest possible $c$ here is denoted by $\|w\|_{A_1}$.

It was proved in \cite{LOP} that for any $w\in
A_1$,
\begin{equation}\label{sha1}
\|Tf\|_{L^p(w)}\le
c(n,T)pp'\|w\|_{A_1}\|f\|_{L^p(w)}\quad(1<p<\infty)
\end{equation}
and
\begin{equation}\label{sha2}
\|Tf\|_{L^{1,\infty}(w)}\le
c(n,T)\|w\|_{A_1}\log(1+\|w\|_{A_1})\|f\|_{L^1(w)}.
\end{equation}

The so-called weak Muckenhoupt-Wheeden conjecture says that (\ref{sha2}) holds with the linear dependence on $\|w\|_{A_1}$.
However, this was recently disproved in \cite{NRVV}, and it raises a conjecture that the $L\log L$ dependence on $\|w\|_{A_1}$
in (\ref{sha2}) is best possible.

Very recently, both estimates (\ref{sha1}) and (\ref{sha2}) have been improved by T. Hyt\"onen and C. P\'erez \cite{HP} as follows:
\begin{equation}\label{sha3}
\|Tf\|_{L^p(w)}\le
c(n,T)pp'\|w\|_{A_1}^{1/p}\|w\|_{A_{\infty}}^{1/p'}\|f\|_{L^p(w)}\quad(1<p<\infty)
\end{equation}
and
\begin{equation}\label{sha4}
\|Tf\|_{L^{1,\infty}(w)}\le
c(n,T)\|w\|_{A_1}\log(1+\|w\|_{A_{\infty}})\|f\|_{L^1(w)}.
\end{equation}

Note that (\ref{sha4}) follows from (\ref{sha3}) by means of the Calder\'on-Zygmund method. Inequality (\ref{sha3}) was deduced in \cite{HP} from
a sharp version of the reverse H\"older inequality along with the estimate
\begin{equation}\label{intest}
\|Tf\|_{L^p(w)}\le c(n,T)pp'\Big(\frac{1}{r-1}\Big)^{1-1/pr}\|f\|_{L^p(M_rw)}\quad (1<r<2).
\end{equation}
proved in \cite{LOP} (here $M_rw=M(w^r)^{1/r}$). The method of the proof of (\ref{intest}) leaves open a question whether this inequality (and so (\ref{sha3}) and (\ref{sha4})) holds for the maximal Calder\'on-Zygmund operator $T_{\natural}$ as well. Theorem \ref{mainr} yields a positive answer to this question.

\begin{theorem}\label{a1es}
Inequalities (\ref{sha3}) and (\ref{sha4}) remain true for $T_{\natural}$ in place of $T$.
\end{theorem}

It follows from the discussion above that it suffices to prove (\ref{intest}) for $T_{\natural}$. The rest of the argument is exactly the same as in \cite{HP}.
Next, by Theorem \ref{mainr}, it suffices to prove (\ref{intest}) for ${\mathcal A}$. This can be done in a variety of ways. For example, it was shown in \cite{L1}
that (\ref{intest}) follows from
$$\int_{{\mathbb R}^n}|Tf|w\,dx\le c(n,T)\int_{{\mathbb
R}^n}(Mf)^{\d}M\big((Mf)^{1-\d}w\big)\,dx\quad(0<\d\le 1).$$
Exactly as in \cite{L1} we have that this inequality with ${\mathcal A}$ in place of $T$ would imply (\ref{intest}) for ${\mathcal A}$. But this is almost trivial:
\begin{eqnarray*}
\int_{{\mathbb R}^n}({\mathcal A}f)w\,dx&=&\sum_{k,j}f_{Q_j^k}w_{Q_j^k}|Q_j^k|\le 2\sum_{k,j}(f_{Q_j^k})^{\d}((Mf)^{1-\d}w)_{Q_j^k}|E_j^k|\\
&\le& 2\sum_{j,k}\int_{E_j^k}(Mf)^{\d}M((Mf)^{1-\d}w)dx\\
&\le& 2\int_{{\mathbb
R}^n}(Mf)^{\d}M\big((Mf)^{1-\d}w\big)\,dx.
\end{eqnarray*}

\section{Haar shift operators}
We recall briefly main definitions concerning Haar shift operators. For more details we refer to \cite{H,HLMORSU,HPTV}.

\begin{definition}\label{hf}
We say that $h_Q$ is a Haar function on a cube $Q\in {\mathscr{D}}$ if
\begin{enumerate}
\renewcommand{\labelenumi}{(\roman{enumi})}
\item
$h_Q$ is a function supported on $Q$, and is constant on the
children of $Q$;
\item $\int h_Q=0$;
\end{enumerate}

We say that $h_Q$ is a generalized Haar function if it is a linear
combination of a Haar function on $Q$ and $\chi_Q$ (in other words,
only condition (i) above is satisfied).
\end{definition}

\begin{definition}
Given a general dyadic grid ${\mathscr{D}}$, $(m,k)\in {\mathbb
Z}_+^2$, and $Q\in {\mathscr{D}}$, set
$${\mathbb S}_Qf(x)=\sum_{{Q',Q''\in
{\mathscr{D}},Q',Q''\subset Q}\atop
{\ell(Q')=2^{-m}\ell(Q),\ell(Q'')=2^{-k}\ell(Q)}}\frac {\langle
f,h_{Q'}^{Q''}\rangle}{|Q|}h_{Q''}^{Q'}(x),
$$
where $h_{Q'}^{Q''}$ is a (generalized) Haar function on $Q'$, and
$h_{Q''}^{Q'}$ is one on $Q''$ such that
$$\|h_{Q'}^{Q''}\|_{L^{\infty}}\|h_{Q''}^{Q'}\|_{L^{\infty}}\le 1.$$
We say that ${\mathbb S}$ is a (generalized) Haar shift
operator of complexity type $(m,k)$ if
$${\mathbb S}f(x)={\mathbb S}_{{\mathscr{D}}}^{m,k}f(x)=
\sum_{Q\in {\mathscr{D}}}{\mathbb S}_Qf(x).$$
The number $\kappa=\max(m,k,1)$ is called the complexity of ${\mathbb S}$.
\end{definition}

Also, it is assumed by the definition the $L^2$ boundedness of the generalized Haar shift operator
(for the usual Haar shift this follows automatically from its properties).

\begin{definition} Given a generalized Haar shift ${\mathbb S}$, define its
associated maximal truncations by
$${\mathbb S}_{\natural}f(x)=\sup_{0<\e\le v<\infty}|{\mathbb S}_{\e,v}f(x)|,$$
where
$${\mathbb S}_{\e,v}f(x)=\sum_{Q\in {\mathscr{D}}:\e\le{\ell_Q}\le v}{\mathbb S}_Qf(x).$$
\end{definition}

The importance of the defined objects follows from the following result proved by T. Hyt\"onen \cite{H}
and simplified in \cite{HPTV}.

\begin{theorem}\label{repr} Let $T$ be a Calder\'on-Zygmund operator
which satisfies the standard estimates with $\d\in (0,1]$. Then for
all bounded and compactly supported functions $f$ and $g$,
$$
\langle g,Tf\rangle=c(T,n){\mathbb
E}_{\mathscr{D}}\sum_{k,m=0}^{\infty}2^{-(m+k)\d/2}\langle
g,{\mathbb S}_{{\mathscr{D}}}^{m,k}f\rangle,
$$
where ${\mathbb E}_{\mathscr{D}}$ is the expectation with respect to
a probability measure on the space of all general dyadic grids.
\end{theorem}

By means of Theorem \ref{repr}, it was deduced in \cite{HLMORSU} the following estimate.

\begin{prop}\label{point} We have the pointwise bound
$$T_{\natural}f(x)\le c(T,n)\Big(Mf(x)+{\mathbb
E}_{\mathscr{D}}\sum_{k,m=0}^{\infty}2^{-(m+k)\d/2}({\mathbb S}_{{\mathscr{D}}}^{m,k})_{\natural}f(x)\Big).$$
\end{prop}

\section{A ``local mean oscillation decomposition"}

\begin{definition} The non-increasing rearrangement of a measurable function $f$ on ${\mathbb R}^n$ is defined by
$$f^*(t)=\inf\{\a>0:|\{x\in {\mathbb R}^n:|f(x)|<\a\}|<t\}\quad(0<t<\infty).$$
\end{definition}

\begin{definition}
Given a measurable function $f$ on ${\mathbb R}^n$ and a cube $Q$,
the local mean oscillation of $f$ on $Q$ is defined by
$$\o_{\la}(f;Q)=\inf_{c\in {\mathbb R}}
\big((f-c)\chi_{Q}\big)^*\big(\la|Q|\big)\quad(0<\la<1).$$
\end{definition}

\begin{definition}
By a median value of $f$ over $Q$ we mean a possibly nonunique, real
number $m_f(Q)$ such that
$$\max\big(|\{x\in Q: f(x)>m_f(Q)\}|,|\{x\in Q: f(x)<m_f(Q)\}|\big)\le |Q|/2.$$
\end{definition}

It is easy to see that the set of all median values of $f$ is either one point or the closed interval. In the latter case we will assume for
the definiteness that $m_f(Q)$ is the {\it maximal} median value. Observe that it follows from the definitions that
\begin{equation}\label{pro1}
|m_f(Q)|\le (f\chi_Q)^*(|Q|/2).
\end{equation}
This estimate implies
\begin{equation}\label{pro2}
((f-m_f(Q))\chi_Q)^*(\la|Q|)\le 2\o_{\la}(f;Q)\quad(0<\la\le 1/2).
\end{equation}
We also mention that (cf. \cite[Lemma 2.2]{F2})
\begin{equation}\label{pro3}
\lim_{|Q|\to 0,Q\ni x}m_f(Q)=f(x)\quad\text{for a.e.}\,\,\, x\in
{\mathbb R}^n.
\end{equation}

Given a cube $Q_0$, denote by ${\mathcal D}(Q_0)$ the set of all
dyadic cubes with respect to $Q_0$. The dyadic local sharp maximal
function $M^{\#,d}_{\la;Q_0}f$ is defined by
$$M^{\#,d}_{\la;Q_0}f(x)=\sup_{x\in Q'\in
{\mathcal D}(Q_0)}\o_{\la}(f;Q').$$

The following theorem was proved in \cite{L1}.

\begin{theorem}\label{decom} Let $f$ be a measurable function on
${\mathbb R}^n$ and let $Q_0$ be a fixed cube. Then there exists a
(possibly empty) sparse family of cubes $Q_j^k\in {\mathcal D}(Q_0)$
such that for a.e. $x\in Q_0$,
$$
|f(x)-m_f(Q_0)|\le
4M_{1/4;Q_0}^{\#,d}f(x)+4\sum_{k,j}
\o_{\frac{1}{2^{n+2}}}(f;(Q_j^k)^{(1)})\chi_{Q_j^k}(x).
$$
\end{theorem}

Here we will prove a similar result with the local mean oscillations taken over the cubes $Q_j^k$
instead of $(Q_j^k)^{(1)}$.

\begin{theorem}\label{decom1} Let $f$ be a measurable function on
${\mathbb R}^n$ and let $Q_0$ be a fixed cube. Then there exists a
(possibly empty) sparse family of cubes $Q_j^k\in {\mathcal D}(Q_0)$
such that for a.e. $x\in Q_0$,
$$
|f(x)-m_f(Q_0)|\le
4M_{\frac{1}{2^{n+2}};Q_0}^{\#,d}f(x)+2\sum_{k,j}
\o_{\frac{1}{2^{n+2}}}(f;Q_j^k)\chi_{Q_j^k}(x).
$$
\end{theorem}

The key element of the proof is the following.

\begin{lemma}\label{ingr} There exists a (possibly empty) collection of pairwise disjoint cubes
$\{Q_j^1\}\in {\mathcal D}(Q_0)$ such that $\sum_j|Q_j^1|\le \frac{1}{2}|Q_0|$
and for a.e. $x\in Q_0$,
\begin{equation}\label{1st}
f-m_f(Q_0)=g_1+\sum_j\a_{j,1}\chi_{Q_j^1}+\sum_j(f-m_f(Q_j^1))\chi_{Q_j^1},
\end{equation}
where
$
|g_1|\le 2M_{\frac{1}{2^{n+2}};Q_0}^{\#,d}f$ for a.e.
$x\in Q_0\setminus \cup_jQ_j^1$
and the numbers $\a_{j,1}$ satisfy
$|a_{j,1}|\le 2\o_{\frac{1}{2^{n+2}}}(f;Q_0).$
\end{lemma}

Having this lemma established, the proof of
Theorem \ref{decom1} follows exactly the same lines as the proof of Theorem \ref{decom}. Therefore, we only outline briefly main details.

\begin{proof}[Proof of Theorem \ref{decom1}]
Iterating (\ref{1st}) for each $Q_j^1$ and for every subsequent cube, we get that for a.e. $x\in Q_0$,
\begin{equation}\label{summ}
f-m_f(Q_0)=g+\sum_j\a_{j,1}\chi_{Q_j^1}+\sum_{k=2}^{\infty}\,\,\sum_{i:Q_i^{k-1}\cap\Omega_{k}\not=\emptyset}\,\,\sum_{j:Q_j^k\subset Q_i^{k-1}}\alpha_{j,k}^{(i)}\chi_{Q_j^{k}},
\end{equation}
where $\Omega_{k}=\cup_jQ_j^{k}$, and the family $\{Q_j^{k}\}$ is sparse. Moreover,
$$|g|\le 2M_{\frac{1}{2^{n+2}};Q_0}^{\#,d}f\quad\text{and}\quad |\alpha_{j,k}^{(i)}|\le 2\o_{\frac{1}{2^{n+2}}}(f;Q_i^{k-1}).$$

The first sum in (\ref{summ}) is bounded by $2\o_{\frac{1}{2^{n+2}}}(f;Q_0)\le 2M_{\frac{1}{2^{n+2}};Q_0}^{\#,d}f$. Further,
$$\sum_{j:Q_j^k\subset Q_i^{k-1}}|\alpha_{j,k}^{(i)}|\chi_{Q_j^{k}}\le 2\o_{\frac{1}{2^{n+2}}}(f;Q_i^{k-1})\chi_{Q_i^{k-1}}.$$
Hence, the second sum in (\ref{summ}) is bounded by
$$2\sum_{k\ge 2}\sum_i\o_{\frac{1}{2^{n+2}}}(f;Q_i^{k-1})\chi_{Q_i^{k-1}}.$$
Combining the obtained estimates completes the proof.
\end{proof}

\begin{proof}[Proof of Lemma \ref{ingr}]
Set $f_1(x)=f(x)-m_f(Q_0)$ and
$$E_1=\{x\in Q_0:|f_1(x)|>(f_1\chi_{Q_0})^*(\la_n|Q_0|)\},$$
where $\la_n=\frac{1}{2^{n+2}}$. If $|E_1|=0$, then by (\ref{pro2}) we trivially have
$$|f(x)-m_f(Q_0)|\le 2\o_{\la_n}(f;Q_0)\le 2M^{\#,d}_{\la_n;Q_0}f(x)
\quad\text{for a.e.}\,\,x\in Q_0.$$

Assume therefore that $|E_1|>0$. Let
$$
{\mathfrak{m}}_{Q_0}f_1(x)=\sup_{x\in Q\in {\mathcal D}(Q_0)}\max_{Q_i:Q_i^{(1)}=Q}|m_{f_1}(Q_i)|
$$
(the maximum is taken over $2^n$ dyadic children of $Q$). Consider the set
$$\O_1=\{x\in Q_0:{\mathfrak{m}}_{Q_0}f_1(x)>(f_1\chi_{Q_0})^*(\la_n|Q_0|)\}.$$
By (\ref{pro3}), ${\mathfrak{m}}_{Q_0}f_1(x)\ge |f_1(x)|$ a.e., and hence $|\O_1|\ge |E_1|>0$.
We can write $\O_1=\cup Q_j^1$, where $Q_j^1$ are pairwise disjoint cubes from ${\mathcal D}(Q_0)$
with the property that they are maximal such that
\begin{equation}\label{pr}
\max_{Q_i:Q_i^{(1)}=Q_j^1}|m_{f_1}(Q_i)|>(f_1\chi_{Q_0})^*(\la_n|Q_0|).
\end{equation}
In particular, this means that each $Q_j^1$ satisfies
$$
|m_{f_1}(Q_j^1)|\le (f_1\chi_{Q_0})^*(\la_n|Q_0|)\le 2\o_{\la_n}(f;Q_0).
$$

Since $m_{f_1}(Q_j^1)=m_{f}(Q_j^1)-m_f(Q_0)$, we have
$$
f-m_f(Q_0)=f_1\chi_{Q_0\setminus \O_1}+\sum_jm_{f_1}(Q_j^1)\chi_{Q_j^1}+\sum_j(f-m_f(Q_j^1))\chi_{Q_j^1},
$$
which proves (\ref{1st}) with $g_1=f_1\chi_{Q_0\setminus \O_1}$ and $\a_{j,1}=m_{f_1}(Q_j^1)$.
By the above established properties we have that $g_1$ and $\a_{j,1}$ satisfy the statement of the lemma.

It remains to show that $|\O_1|\le |Q_0|/2$. If
$Q_i$ is a child of $Q$, then by (\ref{pro1}),
$$|m_f(Q_i)|\le (f\chi_{Q_i})^*(|Q_i|/2)\le (f\chi_Q)^*(|Q|/2^{n+1}).$$
Therefore, if (\ref{pr}) holds, then
$$(f_1\chi_{Q_0})^*(\la_n|Q_0|)<(f_1\chi_{Q_j^1})^*(|Q_j^1|/2^{n+1}).$$
Hence,
$$|\{x\in Q_j^1:|f_1(x)|>(f_1\chi_{Q_0})^*(\la_n|Q_0|)\}|\ge
|Q_j^1|/2^{n+1},$$ and thus
\begin{eqnarray*}
\frac{1}{2^{n+1}}\sum_j|Q_j^1|&\le& \sum_j |\{x\in
Q_j^1:|f_1(x)|>(f_1\chi_{Q_0})^*(\la_n|Q_0|)\}|\\
&\le& |\{x\in Q_0:|f_1(x)|>(f_1\chi_{Q_0})^*(\la_n|Q_0|)\}|\le \la_n|Q_0|,
\end{eqnarray*}
which completes the proof.
\end{proof}

\section{Proof of Theorem \ref{mainr}}
Taking into account Proposition \ref{point}, in order to prove Theorem \ref{mainr}, it suffices to show that
\begin{equation}\label{esn1}
\|Mf\|_{X}\le c(n)\|{\mathcal A}f\|_{X}\quad(f\ge 0)
\end{equation}
and
\begin{equation}\label{esn2}
\|({\mathbb S}_{{\mathscr{D}}}^{m,k})_{\natural}f\|_{X}\le c(n)\kappa^2\|{\mathcal A}|f|\|_{X}.
\end{equation}

Here and below, $\|{\mathcal A}_if\|_{X}$ is understood as $\displaystyle\sup_{{\mathscr{D}},{\mathcal S}}\|{\mathcal A}_{{\mathscr{D}},{\mathcal S},i}f\|_{X}$, where the supremum is taken over arbitrary dyadic grids ${\mathscr{D}}$ and sparse families ${\mathcal S}\in {\mathscr{D}}$.

\subsection{Banach function spaces}
For a general account of Banach function spaces we refer to
\cite[Ch. 1]{BS}. Here we mention only several facts which will be
used below.

The associate space $X'$ consists of measurable functions $f$ for which
$$\|f\|_{X'}=\sup_{\|g\|_{X}\le 1}\int_{{\mathbb R}^n}|f(x)g(x)|dx<\infty.$$
This definition implies the following H\"older inequality:
\begin{equation}\label{hol}
\int_{{\mathbb R}^n}|f(x)g(x)|dx\le \|f\|_{X}\|g\|_{X'}.
\end{equation}
Further \cite[p. 13]{BS},
\begin{equation}\label{dua}
\|f\|_{X}=\sup_{\|g\|_{X'}=1}\int_{{\mathbb R}^n}|f(x)g(x)|dx.
\end{equation}

By Fatou's lemma \cite[p. 5]{BS}, if $f_n\to f$ a.e., and if
$\displaystyle\liminf_{n\to\infty}\|f_n\|_{X}<\infty$, then $f\in
X$, and
\begin{equation}\label{fatou}
\|f\|_X\le \liminf_{n\to \infty}\|f_n\|_X.
\end{equation}

\subsection{Proof of (\ref{esn1})}
We shall use the well-known principle saying that in order to estimate the usual maximal operator it suffices to estimate the dyadic one.
This principle has several forms. We shall need the one attributed in the literature to M. Christ and, independently, to J.~Garnett and P. Jones. However, we have found it in a very clear form only in \cite[proof of Th. 1.10]{HP}.

\begin{prop}\label{prhp} There are $2^n$ dyadic grids ${\mathscr{D}}_{\a}$ such that for any cube $Q\subset {\mathbb R}^n$ there exists a cube $Q_{\a}\in {\mathscr{D}}_{\a}$
such that $Q\subset Q_{\a}$ and $\ell_{Q_{\a}}\le 6\ell_Q$.
\end{prop}

It follows from this Proposition that
\begin{equation}\label{interm}
Mf(x)\le 6^n\sum_{\a=1}^{2^n}M^{{\mathscr{D}}_{\a}}f(x).
\end{equation}

By the Calder\'on-Zygmund decomposition, if $\{x:M^df(x)>2^{(n+1)k}\}=\cup_{j}{Q_j^k}$, then the family $\{Q_j^k\}$ is sparse and
$$M^df(x)\le 2^{n+1}\sum_{k,j}f_{Q_j^k}\chi_{E_j^k}(x)\le 2^{n+1}{\mathcal A}f(x).$$
From this and from (\ref{interm}),
\begin{equation}\label{pb}
Mf(x)\le 2\cdot 12^n\sum_{\a=1}^{2^n}{\mathcal A}_{{\mathscr{D}}_{\a},{\mathcal S}_{\a}}f(x),
\end{equation}
where ${\mathcal S}_{\a}\in {\mathscr{D}}_{\a}$ depends on $f$.
This implies (\ref{esn1}) with $c(n)=2\cdot 24^n$.

\subsection{Proof of (\ref{esn2})}
We start with the following lemma by T. Hyt\"onen and M.~Lacey \cite{HL}.

\begin{lemma}\label{HL} If ${\mathbb S}$ has complexity $\kappa$, then for any dyadic $Q$
$$\o_{\la}({\mathbb S}_{\natural}f;Q)\le c(\la,n)\Big(\kappa |f|_Q+\sum_{i=1}^{\kappa}|f|_{Q^{(i)}}\Big).$$
\end{lemma}

Observe that ``dyadic" here means that $Q\in {\mathscr{D}}$ if ${\mathbb S}={\mathbb S}_{\mathscr{D}}$.
Combining  Lemma \ref{HL} with Theorem \ref{decom1}, we get
$$|({\mathbb S}_{{\mathscr{D}}}^{m,k})_{\natural}f(x)-m_{({\mathbb S}_{{\mathscr{D}}}^{m,k})_{\natural}f}(Q_0)|\le c(n)
\Big(\kappa Mf(x)+\kappa {\mathcal A}|f|(x)+\sum_{i=1}^{\kappa}{\mathcal A}_i|f|(x)\Big).$$
Assuming that $f$ is bounded and with compact support,
we have by (\ref{pro1}) that $m_{({\mathbb S}_{{\mathscr{D}}}^{m,k})_{\natural}f}(Q)\to 0$ as $Q$ expands unboundedly. Therefore, the previous inequality
combined with Fatou's lemma (\ref{fatou}) implies
$$\|({\mathbb S}_{{\mathscr{D}}}^{m,k})_{\natural}f\|_{X}\le c(n)\kappa \big(\|Mf\|_{X}+\|{\mathcal A}|f|\|_{X}\big)
+\sum_{i=1}^{\kappa}\|{\mathcal A}_i|f|\|_{X}.$$
From this and from (\ref{esn1}) we get that in order to prove (\ref{esn2}) it suffices to show that
\begin{equation}\label{mains}
\|{\mathcal A}_i|f|\|_{X}\le c(n)i \|{\mathcal A}|f|\|_{X}.
\end{equation}
Exactly as above, one can assume that ${\mathcal A}_i$ is defined by means of the standard dyadic grid. Also, since we shall deal below
only with ${\mathcal A}_i$ and $M$, one can assume that $f\ge 0$.

Consider the formal adjoint of ${\mathcal A}_i$:
$${\mathcal A}_i^{\star}f(x)=\frac{1}{2^{in}}\sum_{k,j}f_{Q_j^k}\chi_{(Q_j^k)^{(i)}}(x).$$
Our goal is to show that the operator ${\mathcal A}_i^{\star}$ is of weak type $(1,1)$ with the bound depending linearly on $i$. This will be done by
the classical Calder\'on-Zygmund argument. Hence we start with the $L^2$ boundedness of ${\mathcal A}_i^{\star}$. In the proof below we use the well known fact that $\|M^d\|_{L^p}\le p'$.

\begin{prop}\label{a2b} For any $i\in{\mathbb N}$,
$$\|{\mathcal A}_i^{\star}f\|_{L^2}=\|{\mathcal A}_if\|_{L^2}\le 8\|f\|_{L^2}.$$
\end{prop}

\begin{proof}
Similarly to the proof of (\ref{estcmp}), we have
\begin{eqnarray*}
\int_{{\mathbb R}^n}({\mathcal A}_if)gdx=\sum_{k,j}f_{(Q_j^k)^{(i)}}g_{Q_j^k}|Q_j^k|&\le& 2\sum_{k,j}\int_{E_j^k}(M^df)(M^dg)dx\\
&\le& 2\int_{{\mathbb R}^n}(M^df)(M^dg)dx.
\end{eqnarray*}
From this, using H\"older's inequality, the $L^2$ boundedness of $M^d$ and duality, we get the $L^2$ bound for ${\mathcal A}_i$.
\end{proof}

\begin{lemma}\label{weaktype} For any $i\in{\mathbb N}$,
$$\|{\mathcal A}_i^{\star}f\|_{L^{1,\infty}}\le ci\|f\|_{L^1},$$
where $c$ is an absolute constant (for $i$ big enough one can take $c=5$).
\end{lemma}

\begin{proof} Let $\Omega=\{x:M^df(x)>\a\}=\cup_lQ_l$, where $Q_l$ are maximal pairwise disjoint dyadic cubes such that
$f_{Q_l}>\a$. Set also
$$b_l=(f-f_{Q_l})\chi_{Q_l},\quad b=\sum_lb_l$$
and $g=f-b$. We have
\begin{eqnarray}
|\{x:|{\mathcal A}_i^{\star}f(x)|>\a\}|&\le& |\Omega|+|\{x:|{\mathcal A}_i^{\star}g(x)|>\a/2\}|\nonumber\\
&+&|\{x\in \Omega^c:|{\mathcal A}_i^{\star}b(x)|>\a/2\}|.\label{set}
\end{eqnarray}
Further, $|\Omega|\le \frac{1}{\a}\|f\|_{L_1}$, and, by the $L^2$ boundedness of ${\mathcal A}_i^{\star}$,
$$|\{x:|{\mathcal A}_i^{\star}g(x)|>\a/2\}|\le \frac{4}{\a^2}\|{\mathcal A}_i^{\star}g\|_{L^2}^2\le \frac{c}{\a^2}\|g\|_{L^2}^2\le \frac{c}{\a}\|g\|_{L^1}\le \frac{c}{\a}\|f\|_{L^1}.$$

It remains therefore to estimate the term in (\ref{set}). For $x\in \Omega^c$ consider
$${\mathcal A}_i^{\star}b(x)=\frac{1}{2^{in}}\sum_l\sum_{k,j}(b_l)_{Q_j^k}\chi_{(Q_j^k)^{(i)}}(x).$$
The second sum is taken over those cubes $Q_j^k$ for which $Q_j^k\cap Q_l\not=\emptyset$. If $Q_l\subset Q_j^k$, then
$(b_l)_{Q_j^k}=0$. Therefore one can assume that $Q_j^k\subset Q_l$. On the other hand, if $(Q_j^k)^{(i)}\cap \Omega^c\not=\emptyset$, then
$Q_l\subset (Q_j^k)^{(i)}$. Hence, for $x\in \Omega^c$ we have
$${\mathcal A}_i^{\star}b(x)=\frac{1}{2^{in}}\sum_l\sum_{k,j:Q_j^k\subset Q_l\subset (Q_j^k)^{(i)}}(b_l)_{Q_j^k}\chi_{(Q_j^k)^{(i)}}(x).$$
The latter sum is nontrivial if $i\ge 2$. In this case the family of all dyadic cubes $Q$ for which $Q\subset Q_l\subset Q^{(i)}$ can be decomposed into
$i-1$ families of disjoint cubes of equal length. Therefore,
$$\sum_{k,j:Q_j^k\subset Q_l\subset (Q_j^k)^{(i)}}\chi_{Q_j^k}\le (i-1)\chi_{Q_l}.$$
From this we get
\begin{eqnarray*}
&&|\{x\in \Omega^c:|{\mathcal A}_i^{\star}b(x)|>\a/2\}|\le \frac{2}{\a}\|{\mathcal A}_i^{\star}b\|_{L^1(\Omega^c)}\\
&&\le \frac{2}{\a}\sum_l\sum_{k,j:Q_j^k\subset Q_l\subset (Q_j^k)^{(i)}}\int_{Q_j^k}|b_l|dx\le \frac{2(i-1)}{\a}\sum_l
\int_{Q_l}|b_l|dx\\
&&\le\frac{4(i-1)}{\a}\|f\|_{L^1}.
\end{eqnarray*}
The proof is complete.
\end{proof}

\begin{lemma}\label{locosc} Let $i\in {\mathbb N}$. For any dyadic cube $Q$,
$$\o_{\la_n}({\mathcal A}_i^{\star}f;Q)\le c(n)if_Q.$$
\end{lemma}

\begin{proof} For $x\in Q$,
$$\frac{1}{2^{in}}\sum_{k,j:Q\subseteq (Q_j^k)^{(i)}}f_{Q_j^k}\chi_{(Q_j^k)^{(i)}}(x)=\frac{1}{2^{in}}\sum_{k,j:Q\subseteq (Q_j^k)^{(i)}}f_{Q_j^k}\equiv c.$$
Hence
$$|{\mathcal A}_i^{\star}f(x)-c|\chi_Q(x)=\frac{1}{2^{in}}\sum_{k,j:(Q_j^k)^{(i)}\subset Q}f_{Q_j^k}\chi_{(Q_j^k)^{(i)}}(x)\le {\mathcal A}_i^{\star}(f\chi_Q)(x).$$
From this and from Lemma \ref{weaktype},
$$\inf_c(({\mathcal A}_i^{\star}f-c)\chi_Q)^*(\la_n|Q|)\le ({\mathcal A}_i^{\star}(f\chi_Q))^*(\la_n|Q|)\le c(n)if_Q,$$
which completes the proof.
\end{proof}

We are ready now to prove (\ref{mains}). By the standard limiting argument, one can assume that the sum defining ${\mathcal A}_i$ is finite.
Then $m_{{\mathcal A}_i^{\star}f}(Q)=0$ for $Q$ big enough. Hence, By Lemma \ref{locosc} and Theorem \ref{decom1}, for a.e. $x\in~Q$,
$${\mathcal A}_i^{\star}f(x)\le c(n)i\big(Mf(x)+{\mathcal A}f(x)\big)$$
(notice that here ${\mathcal A}_i$ and ${\mathcal A}$ are taken with respect to different sparse families).
From this and from (\ref{pb}), and using that the operator ${\mathcal A}$ is self-adjoint, for any $g\ge 0$ we have
\begin{eqnarray*}
\int_{{\mathbb R}^n}({\mathcal A}_if)gdx&=&\int_{{\mathbb R}^n}f({\mathcal A}_i^{\star}g)dx\\
&\le& c_n i\sum_{\a=1}^{2^n+1}\int_{{\mathbb R}^n}f({\mathcal A}_{{\mathscr{D}}_{\a},{\mathcal S}_{\a}}g)dx\\
&=&c_n i\sum_{\a=1}^{2^n+1}\int_{{\mathbb R}^n}({\mathcal A}_{{\mathscr{D}}_{\a},{\mathcal S}_{\a}}f)gdx\le c'_n i\sup_{{\mathscr{D}},{\mathcal S}}
\|{\mathcal A}_{{\mathscr{D}},{\mathcal S}}f\|_X\|g\|_{X'}.
\end{eqnarray*}
Applying (\ref{dua}) yields (\ref{mains}), and therefore the proof of Theorem \ref{mainr} is complete.

\section{Proof of Theorems \ref{difch} and \ref{twosuf}}

\begin{proof}[Proof of Theorem \ref{difch}]
Using the same argument as in the proof of (\ref{estcmp}), we have
\begin{eqnarray*}
&&\int_{{\mathbb R}^n}({\mathcal A}f)g=\sum_{j,k}f_{Q_j^k}g_{Q_j^k}|Q_j^k|
\le 2\sum_{j,k}\int_{E_j^k}\mathcal M(f,g)dx\\
&&\le 2\int_{{\mathbb R}^n}\mathcal M(f,g)dx\le 2
\|{\mathcal M}\|_{L^p(v)\times L^{p'}(u^{1-p'})\to L^1}\|f\|_{L^p(v)}\|g\|_{L^{p'}(u^{1-p'})}.
\end{eqnarray*}
Taking the supremum over $g$ with $\|g\|_{L^{p'}(u^{1-p'})}=1$ gives
$$
\|{\mathcal A}\|_{L^p(v)\to L^p(u)}\le 2
\|{\mathcal M}\|_{L^p(v)\times L^{p'}(u^{1-p'})\to L^1}.
$$

On the other hand, by Proposition \ref{prhp},
\begin{equation}\label{fg}
{\mathcal M}(f,g)(x)\le 12^n\sum_{\a=1}^{2^n}{\mathcal M}^{{\mathscr{D}}_{\a}}(f,g)(x).
\end{equation}
Consider ${\mathcal M}^d(f,g)$ taken with respect to the standard dyadic grid. Suppose that $f,g\ge 0$ and $f,g\in L^1$.  We will use
exactly the same argument as in the Calder\'on-Zygmund decomposition. For $c_n$ which will be specified below and for
$k\in {\mathbb Z}$ consider the sets
$$\Omega_k=\{x\in {\mathbb R}^n: {\mathcal M}^d(f,g)(x)>c_n^k\}.$$
Then we have that $\Omega_k=\cup_jQ_j^k$, where the cubes $Q_j^k$ are pairwise disjoint with $k$ fixed, and
$$c_n^k<f_{Q_j^k}g_{Q_j^k}\le 2^{2n}c_n^k.$$
From this and from H\"older's inequality,
\begin{eqnarray*}
|Q_j^k\cap \Omega_{k+1}|&=&\sum_{Q_i^{k+1}\subset Q_j^k}|Q_i^{k+1}|\\
&<&c_n^{-\frac{k+1}{2}}\sum_{Q_i^{k+1}\subset Q_j^k}\left(\int_{Q_i^{k+1}}f\int_{Q_i^{k+1}}g\right)^{1/2}\\
&\le& c_n^{-\frac{k+1}{2}}\left(\int_{Q_j^{k}}f\int_{Q_j^{k}}g\right)^{1/2}\le 2^nc_n^{-1/2}|Q_j^k|
\end{eqnarray*}
Hence, taking $c_n=2^{2(n+1)}$, we obtain that the family $\{Q_j^k\}$ is sparse, and
$${\mathcal M}^d(f,g)(x)\le 2^{2(n+1)}\sum_{j,k}f_{Q_j^k}g_{Q_j^k}\chi_{Q_j^k}(x).$$
Therefore,
$$\int_{{\mathbb R}^n}{\mathcal M}^d(f,g)dx\le 2^{2(n+1)}\sum_{j,k}f_{Q_j^k}g_{Q_j^k}|Q_j^k|=2^{2(n+1)}\int_{{\mathbb R}^n}({\mathcal A}f)gdx.$$
From this and from (\ref{fg}), applying H\"older's inequality, we obtain
\begin{eqnarray*}
&&\int_{{\mathbb R}^n}{\mathcal M}(f,g)(x)dx\le 4\cdot 48^n\sum_{\a=1}^{2^n}\int_{{\mathbb R}^n}({\mathcal A}_{{\mathscr{D}}_{\a},{\mathcal S}_{\a}}f)gdx\\
&&\le 4\cdot 48^n\sum_{\a=1}^{2^n}\|{\mathcal A}_{{\mathscr{D}}_{\a},{\mathcal S}_{\a}}f\|_{L^p(u)}\|g\|_{L^{p'}(u^{1-p'})}\\
&&\le 4\cdot 96^n\sup_{{\mathscr{D}},{\mathcal S}}\|{\mathcal A}_{{\mathscr{D}},{\mathcal S}}\|_{L^p(v)\to L^p(u)}\|f\|_{L^p(v)}\|g\|_{L^{p'}(u^{1-p'})},
\end{eqnarray*}
which completes the proof.
\end{proof}

\begin{proof}[Proof of Theorem \ref{twosuf}]
By H\"older's inequality (\ref{hol}),
$$
|f|_Q|g|_Q\le \|fv^{1/p}\|_{Y',Q}\|v^{-1/p}\|_{Y,Q}\|gu^{-1/p}\|_{X',Q}\|u^{1/p}\|_{X,Q}.
$$
Hence,
$$
{\mathcal M}(f,g)(x)\le c(u,v)M_{Y'}(fv^{1/p})(x)M_{X'}(gu^{-1/p})(x),
$$
where $c(u,v)=\sup_Q\|u^{1/p}\|_{X,Q}\|v^{-1/p}\|_{Y,Q}$.
Therefore, by assumptions on $X'$ and $Y'$ and by the usual H\"older's inequality,
\begin{eqnarray*}
\int_{{\mathbb R}^n}{\mathcal M}(f,g)(x)dx&\le& c(u,v)\|M_{Y'}(fv^{1/p})\|_{L^p}\|M_{X'}(gu^{-1/p})\|_{L^{p'}}\\
&\le& c(u,v)\|f\|_{L^p(v)}\|g\|_{L^{p'}(u^{1-p'})}.
\end{eqnarray*}
Combining this estimate with Theorems \ref{mainr} and \ref{difch} completes the proof.
\end{proof}

\vskip 0.5cm
\noindent
{\bf Added in proof.} We have found \cite{L3} that the main result of this paper can be proved without the use of
the Haar shift operators. This further simplifies the proof of the $A_2$ conjecture.
Almost simultaneously, a proof of the $A_2$ conjecture based on a similar idea was obtained in \cite{HLP}.

\end{document}